\documentclass[12pt, leqno]{amsart}

    \usepackage{indentfirst}
    \usepackage{amstext}
    \usepackage{amsopn}
    \usepackage{amsfonts}
    \usepackage{amsmath}
    \usepackage{latexsym}
    \usepackage{amscd}
    \usepackage{amssymb}
    \usepackage{amsmath}
    \usepackage[all,cmtip]{xy}
    \usepackage{graphicx}
    \usepackage{tikz}

    =\oddsidemargin \font\teneufm=eufm10 \font\seveneufm=eufm7
    \font\fiveeufm=eufm5
    \newfam\eufmfam
    \textfont\eufmfam=\teneufm \scriptfont\eufmfam=\seveneufm
    \scriptscriptfont\eufmfam=\fiveeufm

    \def\1{\mbox{\bf 1}}

    \newtheorem{theorem}{Theorem}[subsection]

    \newtheorem{lemma}[theorem]{Lemma}

    \theoremstyle{definition}
    \newtheorem{remark}[theorem]{Remark}

    \newtheorem{definition}[theorem]{Definition}

    \numberwithin{equation}{section}
    
    \def\2int{\mathop{2\int}\nolimits}

    \def\resp.{\mathop{\rm resp.}\nolimits}

    \font\math=cmmi10
    \def\varpi{\hbox{\math\char'44}}

    \def\pa{\S\kern.15em }

    \def\un{\uppercase\expandafter{\romannumeral 1}}
    \def\deux{\uppercase\expandafter{\romannumeral 2}}
    \def\trois{\uppercase\expandafter{\romannumeral 3}}
    \def\quatre{\uppercase\expandafter{\romannumeral 4}}
    \def\cinq{\uppercase\expandafter{\romannumeral 5}}
    \def\six{\uppercase\expandafter{\romannumeral 6}}

    \def\hfl#1#2#3{\smash{\mathop{\hbox to#3{\rightarrowfill}}\limits
    ^{\scriptstyle#1}_{\scriptstyle#2}}}
    \def\gfl#1#2#3{\smash{\mathop{\hbox to#3{\leftarrowfill}}\limits
    ^{\scriptstyle#1}_{\scriptstyle#2}}}

\begin{document}
    
   \title[Finite Unitary Rings with Non-Solvable Groups of Units]{Finite unitary rings all of whose groups of units of all their subrings except of the ring itself are solvable}

\author{Mohsen Amiri$^1$, Wilhelm Alexander Cardoso Steinmetz$^2$}


    \maketitle

   \medskip

  \footnotesize{$^1$ \textsc{Departmento de Matem\'atica, Instituto de Ci\^encias Exatas}}\par\nopagebreak
  \footnotesize{ \textsc{Universidade Federal do Amazonas (UFAM), Manaus/AM, Brazil}}\par\nopagebreak
  \footnotesize{\textit{E-mail address:} \texttt{mohsen@ufam.edu.br}}

  \medskip

  \footnotesize{$^2$ \textsc{Departmento de Matem\'atica, Instituto de Ci\^encias Exatas}}\par\nopagebreak
  \footnotesize{\textsc{Universidade Federal de Minas Gerais (UFMG), Belo Horizonte/MG, Brazil}}\par\nopagebreak
  \footnotesize{\textit{E-mail address:} \texttt{wacs@ufmg.br}}
 
\normalsize

 \medskip
  \medskip


    \begin{abstract}
        
       Let R be a finite unitary ring whose group of units is not solvable but all groups of units of all its proper subrings are solvable. In this paper we classify these rings and show that all finite rings of order $p^n$ for $n < 5$ and some of order $p^6$ are in this class of rings.

    \end{abstract}

    \section[Introduction]{Introduction}
    One of the classical problems in algebra is the study and classification of certain algebraic structures subject to specific conditions. There are many studies that take this approach. For example, the class of all rings, all of whose subrings are ideals are studied and classified in \cite{Andrijanov}, the class of all rings having few zero divisors are classified in \cite{Corbas}, the class of all rings such that their additive endomorphisms are ring endomorphisms are classified in \cite{Birkenmeier} and the class of all division rings having finite cardinality are classified in \cite{Wedderburn}. More recently, the class of commutative reduced filial rings is studied in \cite{Andru}, the class of groups in which every element has a paracentralizer of finite index is studied in \cite{Falco}, the class of groups having Schur multiplier of maximum order is classified in \cite{Rai}, the class of all finite unitary rings in which all Sylow subgroups of the group of units are cyclic is characterized in \cite{Arian} and  the class of all finite unitary rings in which the group of units is a minimal non-nilpotent group is classified in \cite{Amini}.
    
    Minimal simple groups have been classified in \cite{Thomp} and in the present article we provide a similar classification for finite unitary rings, analyzing the simplicity of their groups of units and the groups of units of their subrings. More precisely, we characterize the structure of finite rings $R$ such that $R^*$ is not a solvable group but $S^*$ is a solvable group for every proper unitary subring $S$ of $R$. A similar classification for rings with minimal non-nilpotent groups of units has been done in \cite{Amini}. As an application of our main result we describe the structure of a finite ring $R$ of order $p^m$, $p$ a prime and $m\leq 6$, such that $R^*$ is not a solvable group.

   \section{Preliminaries}
     
     The cardinality of a given set $X$ is denoted by $|X|$. For $G$ a group we denote the order of an element $g\in G$ by $o(g)$. For a subset $S$ of a group $G$, the centralizer and normalizer of $S$ in $G$ are defined, respectively, by $C_G(S)=\{g \in G\,\, | \,\, gs = sg\,\, \mbox{for  all}\,\,  s\in S\}$ and $N_G(S)=\{g\in G \,\, | \,\, Sg=gS\}$. We use analogous definitions for centralizers and normalizers in rings. Unless specified otherwise, all rings are unitary. For a ring $R$ we denote $Jac(R)$ its Jacobson radical and $R^*$ its group of units.   The ring of $n\times n$ matrices over a ring $R$ is denoted by $M_n(R)$ and the ring of integers modulo $m$ is denoted by $\mathbb{Z}_m$. The prime subring of $R$ (the subring of $R$ generated ${1}$) is denoted by $R_0$. Let $S$ be a subset of $R$, then the subring generated by $S$ in $R$ is denoted by $R_0[S]$. We denote the characteristic of $R$ by $char(R)$ and $GF(p^m)$ denotes the unique finite field of characteristic $p$ and of order $p^m$. Let us now remind the reader of some results that will be important further on:
    
    \begin{theorem}\label{pr6} (Lemma (1.2) of \cite{Groza})
    Let $R$ be a ring of characteristic $p^r$, where $p$ is a prime number and $r$ is a positive integer. Then the multiplicative group $1 + J(R)$ is a $p$-group of bounded exponent.
    \end{theorem}
    
    \begin{lemma}\label{d}
    Let $R$ be a finite local ring with a nontrivial minimal ideal $I$. Then  $Jac(R) \subseteq ann_R(I)$.
    \end{lemma}
    \begin{proof}
    We have  $Jac(R) I \subseteq I$. Since $I$ is minimal,  either $Jac(R) I = 0$ or $Jac(R) I = I$. As $I\neq 0$, Nakayama's Lemma shows that $Jac(R) I\neq I$. Therefore $Jac(R) I = 0$.
    \end{proof}

    
    \begin{remark}
    Let $R$ be a finite (unitary) ring without any proper unitary subring. Then 
    $R=R_0\cong {\mathbb{Z}}_{n}$ ($n$ some positive integer) is a commutative ring.
     \end{remark}
    
    We now introduce the rings that will be of central interest in this work:

    \begin{definition}
    Let $R$ be a finite (unitary) ring such that $R^*$ is not a solvable group. We say that $R$ is a {\it minimal simple ring}, if $S^*$ is a solvable group for any proper unitary subring $S$ of $R$.
    \end{definition}
    
    \begin{definition} We call a (unitary) ring $R$ {\it purely minimal simple} whenever $R$ is a minimal simple ring of order $p^m$, for some integer $m$, such that $R$ is not a direct sum of two proper ideals.
    \end{definition}
    
    \section{Results}

    We can now turn to proving the results that will help us establish a classification of minimal simple rings. For this, consider the set:
    $$\Delta = \{R, R \bigoplus  {\mathbb{Z}}_{n} \mid n \geq 2 \mbox{ and } R \mbox{ is purely minimal simple} \}.$$
    
    \begin{theorem}
    Let $R$ be a finite minimal simple ring. Then  $R\in \Delta$.
    \end{theorem}

    \proof
    Suppose, for a contradiction, that $R$ is minimal simple of minimal order such that $R\not\in \Delta$. If $|R| \in \{p,p^2\}$, then $R$ is a commutative ring \cite[Lemma]{Eldridge}, so we may suppose that $|R| \notin \{p,p^2\}$. Then, if $S \subsetneq R$, as $R$ is minimal simple, $S^*$ is solvable.
    
    We can write $| R|=p_1^{\alpha_1}p_2^{\alpha_2}...p_k^{\alpha_k}$, where $p_1,...,p_k$ are distinct primes and $k \in \mathbb{N}$. Let us then write $R=R_1 \oplus R_2 \oplus ... \oplus R_k $, where $| R_i | = p_i^{\alpha_i}$ for all $i \in \{ 1,...,k\}$.
    
    We have $R^*=R_1^* \oplus R_2^* \oplus ... \oplus R_k^* $. If $R_i^*$ is solvable  for all $i$,  then $R^*$ is solvable, which is a contradiction. Hence we may assume that $R_1^*$  is not a solvable group.
    If $R_1$ has a proper unitary subring $A_1$ such that $A_1^*$ is not a solvable group, then $(A_1\bigoplus R_2\bigoplus ...\bigoplus R_{k})^*$ is not solvable, which is a contradiction.
    So $R_1$ is a minimal simple ring and by minimality of $R$, $R_1\in \Delta$. Suppose now that for some $i\in \{2,...,k\}$, $R_i$ has a proper unitary subring $B$. Let $$
    H=R_1\bigoplus R_2\bigoplus ...\bigoplus R_{i-1}\bigoplus B\bigoplus R_{i+1}\bigoplus ...\bigoplus R_k.
    $$
    Then $H^*$ would be a proper subgroup of $R^*$, and thus, by assumption, $H^*$ would be solvable. But then $R_1^*$ would also be solvable, which is a contradiction. Thus $R_i\cong \mathbb{Z}_{p_i^{\alpha_i}}$ for all $i\in \{ 2,...,k \}$.
    Since $R_1\in \Delta$, we have $R_1=S$ or  $R_1=S\oplus \mathbb{Z}_{p_1^m}$ where $S$ is a purely minimal simple ring and $m\geq 1$ is an integer. Let $d=|R|/|S|$. It follows from $gcd(p_1,p_2p_3...p_k)=1$ that $ \mathbb{Z}_{p_1^m}\oplus R_2\oplus...\oplus R_k\cong \mathbb{Z}_d$, so  $R\cong S\oplus \mathbb{Z}_d$. Hence $R\in \Delta$, which is a contradiction.
    \endproof
    
    Let now $\Gamma_1$ be the set of all rings $M_2(GF(q))$ such that $q>3$  is a prime number, $\Gamma_2$ be the set of all $M_2(GF(q^p))$ where $q=2,3$ and where $p$ is a prime number and $\Gamma_3=\{M_3(GF(2)),M_3(GF(3))\}$. Write $\Gamma=\Gamma_1\cup\Gamma_2\cup\Gamma_3$.
    
    Let us observe that every element of $\Gamma$ is a minimal simple ring. Suppose $R\in \Gamma$ and let $S$ be a proper subring of $R$. If $R\in \Gamma_1$, then $S$ is of cardinality $q^r$, $r\leq 3$. Thus $S^*$ is a solvable group \cite[Proposition]{Eldridge}. Now let $R\in \Gamma_2\cup \Gamma_3$. Then, by the Artin-Wedderburn structure theorem we have $r\in \{1,2\}$ such that $\frac{S}{Jac(S)}\cong M_r(K)$ for some intermediary field $GF(q)\leq K<GF(q^p)$ if $R\in \Gamma_2$ and $\frac{S}{Jac(S)}\cong M_r(GF(2))$ or $\frac{S}{Jac(S)}\cong M_r(GF(3))$ if $R\in \Gamma_3$. Since $p$ is prime, $K=GF(q)$ and thus in all cases $(\frac{S}{Jac(S)})^*$ is solvable. Since $\frac{S^*}{1+Jac(S)} \cong (\frac{S}{Jac(S)})^*$ and $1+Jac(S)$ is a $p$-group by Lemma \ref{pr6}, $S^*$ is solvable.
    
    To show the converse, we will prove two lemmas. In what follows, let $q=p^r$, where $p$ is a prime and $r\geq 1$. We remind the reader that a Singer cycle of $GL_n(q)$ is an element of order $q^n-1$. These elements were shown to exist in \cite{Singer}, as a result of the fact that the multiplicative group of the field $GF(q^n)$ is cyclic.
    
    \begin{lemma}\label{3211}
    Let $R=M_n(GF(q))\in \Gamma$ with $n=2,3$ and choose $y\in R$ such that $o(y)=q^n-1$ and $\langle y\rangle\cup \{0\}\cong GF(q^n)$. Then:
    \begin{enumerate}
        \item $N_{R^*}(\langle y \rangle) \cong \langle y \rangle \rtimes \langle \sigma \rangle$, where $\sigma$ is the Frobenius autormorphism of $GF(q^n)$ and  $|N_{R^*}(\langle y \rangle)|=n(q^n-1).$
        \item For any $w\in N_{R^*}(\langle y\rangle)\setminus C_{R^*}(y)$,  $(R_0[y,w])^*=(R_0[y^n,w])^*$ is not a solvable group.
        \item If $R=M_3(GF(3))$ and $H\leq R^{*}$ is such that $H\cong S_4$, then the subring generated by $H$ is $R$.
    \end{enumerate}
    
    \end{lemma}
    \begin{proof}
    For (1) see \cite[p.187, Satz 7.3]{Huppert}. We now prove (2). Observe that $\langle y\rangle \subset R_0[y^n,w]$, thus $(R_0[y,w])^*=(R_0[y^n,w])^*$. Since $R\in \Gamma$, $R^*$ is not a solvable group. Let $V:= R_0[y,w]$. Suppose that $c_1+c_2w=0$, for some $c_1,c_2\in F$.
      If $c_2\neq0$, then
    $w=(c_2)^{-1}c_1\in C_{R^*}(y)$, which is a contradiction. So $c_2=0$, and thus $c_1=0$.
    Hence $|V|\geq q^{2n}$. If $n=2$, then $V=R$, so $V^*=R^*$ is not a solvable group.
    
    Suppose now that $n=3$. Then $R=M_3(GF(q))$, $q=2,3$, and $|V|\geq q^6$. Moreover, as $o(y)=q^3-1$, we have $7 \mid V^*$ if $q=2$ and $26 \mid V^*$ if $q=3$. If $V=R$, the $V^*$ is non-solvable, so suppose $V\neq R$. Then $\frac{V}{Jac(V)}\in \{ M_2(GF(q)), M_2(GF(q)) \oplus A\}$, where $A$ is a direct sum of copies of $GF(q)$. But as $(\frac{V}{Jac(V)})^* \cong \frac{V^*}{1+Jac(V)}$, and since $1+Jac(V)$ is trivial or a $q$-group, $V^*$ cannot have an element of order $7$ or $13$, a contradiction.
    
    Finally we prove (3). Let $T=R_0[H]$. Clearly, $-1\in T$, and so 
    $C_2\times H\leq T^*$. Thus $T=R_0[C_2 \times H]=R$, according to a Gap computation. \end{proof}

    \begin{lemma}\label{jj}
    Let $R$ be a local minimal simple ring such that the $I:=Jac(R)$ is a minimal ideal of $R$.
    Suppose that $char(R)=p$ with $p$ prime. Then $\frac{R}{I}\not\in \Gamma$. 
    \end{lemma}
    
    \begin{proof}
    Suppose for a contradiction that $\frac{R}{I}\in \Gamma$. Then  $\frac{R}{I}\cong M_n(GF(q))$ where $M_n(GF(q))\in \Gamma$, so $n\in\{2,3\}$ and $q=p^r$, for some $r\geq 1$. 
    Let $y+I\in \frac{R}{I}$ be such that $\langle y+I \rangle\cup \{0\}\cong GF(q^n)$. By lemma \ref{3211}, there exists $w+I\in N_{(\frac{R}{I})^*}(\langle y+I\rangle)\setminus C_{(\frac{R}{I})^*}(\langle y+I\rangle)$ such that $o(w+I)=n$.
    We consider the following cases:
    
    {\bf Case 1.}
    First suppose that $n=2$ and $p,q>2$. Clearly, $2(q^2-1)\mid \langle y+I\rangle\langle w+I\rangle$. There exists $c\in I$ such that $w^2=1+c$. Then $o(w)\mid 2 q$. Let $z\in \langle w\rangle$ be of order $2$. Therefore $z+I\in N_{(\frac{R}{I})^*}(\langle y+I\rangle)\setminus C_{(\frac{R}{I})^*}(\langle y+I\rangle)$. Replacing $w$ by $z$ if necessary, we may assume that $o(w)=2$. We have $w+I=w(1+I)$ and $y+I=y(1+I)$. Since $gcd(q^2-1,q)=1$, we may assume that $o(y)=q^2-1$. 
    
    Let $K:=((1+I)\rtimes \langle y\rangle)\langle w\rangle$. Since $\frac{K}{1+I}$ is solvable group, and $1+I$ is a $p$-group, $K$ is a solvable group.
    Since $2 (q^2-1)|I|\mid |K|$, and $K$ is a solvable group, $K$ has  Hall subgroup $L$  such that $2(q^2-1)$ and $|L|$ have the same prime divisors.
    We may assume that $y\in L$. Let $z\in L$ of order two  such that $z+I=w+I$. 
    Since  $y^z\in L\cap (1+I)\langle y\rangle=\langle y\rangle$,
    we deduce that $z\in N_{R^*}(\langle y\rangle)\setminus C_{R^*}(y)$.
    Therefore  $zy=y^jz$, for some positive integer $j$. We may assume that $z=w$.
    Let $S:=R_0[y,w]$, and let $Y=\langle y\rangle$.  Then $S=\{\Sigma m_i y^i+\Sigma n_j wy^j \mid m_i,n_j\in GF(p) \mbox{ for } i,j=1,...,q^2-1 \}=GF(p)Y+GF(p)wY$.
    Hence $|S|\leq  p^2(q^{2}-1)^2$.
    From Lemma \ref{3211}, $S=R$, as $R$ is minimal simple. 
    Let $a\in I\setminus\{0\}$. Therefore $\{0,a,ya,...,y^{q^{2}-1}a\}$ is a subset of $I$ of size $q^{2}$. 
    Then $|R|\geq q^{4}q^2>p^2(q^{2}-1)^2$, which is a contradiction.

    {\bf Case 2.} Now suppose that $n=2$ and $p=2$, so $q=2^m$ for some integer $m>1$. We have $y+I=(u+I)(v+I)$, where $o(u+I)=q+1$ and $v+I\in Z(R/I)$. Let $G=PSL(2,q)$ and $h$ be an element of $G$ of order $q+1$. From Theorem 2.1 (3) of \cite{Giudici}, $N_G(\langle h\rangle)$ is isomorphic to the dihedral group $D_{2(q+1)}$.  Hence $\frac{\langle w+I,u+I\rangle}{Z(R/I)}$ is a dihedral group. So $u^w=u^{-1}+s$  for some  $s\in I$. Then $(u^w+u^{-1}c)^2=s^2=0$, and so
    $(u^2)^w+u^{-2}c^2=0$. It follows that $(u^2)^w=-u^{-2}c^2=u^{-2}c^2$. Let $S':=R_0[u^2,w]=R_0[y^2,w]$. By Lemma \ref{3211}, $S'=R$, $R$ being minimal simple. Since $(u^2)^{w^2}=u^2$ and $v+I\in Z(R/I)$,  we have $(y^2)^{w^2}=y^2$, so
    $w^2\in Z(R^*)$.
    
    Then $w^3y^2=wy^2w^2=(wy^2w)w=y^{-2}w=wy^2$.
    Hence $(w^3-w)y^2=0$, and so $w^2-1=0$. It follows that $o(w)=2$.
    Clearly, $|R|=|I|q^4$. Let $a\in I\setminus \{0\}$.
    Let  $Z:=\langle y^2\rangle $. We have $S=GF(2)Z+GF(2)wZ$. Then $|R|=|I|q^4\leq 2^2 (q-1)^2.$
    It follows that $|I|=2$. 
    But $\{0,a,ya,...,y^{q-2}a\}$ is a subset of $I$ of size $q>2$, which is a contradiction.
    
    {\bf Case 3.}
    If $n=3$, then $r=1$ and $p=q\in\{2,3\}$. First suppose that $p=2$.
    Then $o(w+I)=3$ and $o(y+I)=2^3-1=7$.
    By a similar argument to the above, we may assume that $o(w)=3$ and $o(y)=7$. 
    Let  $Y=\langle y\rangle$. We have $w^{-1}Yw=Y$, so $Yw=wY$. Then $S=GF(2)Y+GF(2)wY+GF(2)w^2Y$ is a unitary subring of $R$. Then $|S|\leq 2^3\cdot 7^3$. 
    From Lemma \ref{3211}, $S=R$. 
    Let $a\in I\setminus\{0\}$. Then $\{0,a,ya,...,y^{6}a\}$ is a subset of $I$ of size $8$. It follows that $|R|= |I|\cdot |M_3(GF(p))|\geq 2^3\cdot 2^9>2^3\cdot 7^3$, which is a contradiction.
    
    So suppose that $p=3$. First suppose that $o(w)=3$. Since $gcd(3,26)=1$, we may assume that $o(y)=26$.  Let $Y=\langle y\rangle$ and let $S=R_0[y,w]$, as above, then $S=GF(3)Y+GF(3)wY+GF(3)w^2Y$ is a unitary subring of $R$.  From Lemma \ref{3211}, $S=R$.  Then $|R|\leq 3^3\cdot 26^3$. 
    Let $a\in I\setminus\{0\}$. Then $\{0,a,ya,...,y^{25}a\}$ is a subset of $I$ of size $27$. It follows that $|R|\geq 3^3\cdot 3^9>3^3\cdot 26^3$, which is a contradiction.
    
    So suppose that $o(w)=9$. Then $w^3-1\in I$. According to a Gap computation $|Z(\langle w\rangle \langle y^2\rangle)|=3$ . Since, from Lemma \ref{3211}, $R=R_0[w,y]$, we have $w^3\in Z(R)$.
     Let $a=w^3-1$. Since $a\in Z(R)$, $ann_R(a)$ is a two sided ideal.
     So $ann_R(a)=I$. Let $u,v\in R$ such that $u+I\neq v+I$.
     If $ua=va$, then $u-v\in ann_R(a)=I$, and so $u+I=v+I$, which is a contradiction.
     It follows that $|I|\geq |M_3(GF(3))|\geq 3^9$, and so $|R|\geq 3^9\cdot 3^9=3^{18}$.
    
    Let $H\leq R^*$ such that $\frac{H+I}{I}\cong S_4$. Let $a\in I\setminus\{0\}$. It follows that $\{0,a,ya,...,y^{25}a\}$ is a subset of $I$ of size $27$. Then $|I|\geq 27=3^3$ and since $|1+I|\mid |R^*|$, it follows that $3^3\mid |R^*|$. Since $gcd(3^3,|S_4|)=3$, we have $3^2\cdot |S_4|\mid |R^*|$, and hence $|H|\mid 3^2 \cdot |S_4|$. We have $S_4=((C_2\times C_2)\rtimes C_3)\rtimes C_2$. Let $b\in H$ be of order two and let $W\leq H$ such that $H=W\rtimes \langle b\rangle$. Let $S=R_0[H]=GF(3)W+GF(3)bW$. By Lemma \ref{3211},
    $S=R$. But $|R|\leq 3^2\cdot 3^4 \cdot 12^2 =3^{8}\cdot 2^4<3^{18}$, which is a contradiction. \end{proof}
    
    In the following theorem we classify all finite purely minimal simple unitary rings.
    
    \begin{theorem}\label{1}
    The finite ring $R$ is purely minimal simple if and only if $R\in \Gamma$.
    
    \end{theorem}
    
    \proof
    Let $R$ be a purely minimal simple finite ring. We will show that $R\in \Gamma$. Suppose that $R$ is of minimal cardinality. If $| R | \in \{p,p^2\}$, $p$ a prime, then $R$ is a commutative ring, so suppose that $| R | = p^m$, with $m\geq 3$. We have the following two cases:
    
    {\bf Case 1.}
    Suppose that $Jac(R)=0$.  Then by the Artin-Wedderburn structure theorem, $R\simeq \bigoplus_{i=1}^t M_{n_i}(D_i)$, where $D_i$ is a finite field. Since $R$ is a purely minimal simple ring, we have $t=1$.  So $R\simeq M_{n}(D)$ and since $R$ is not commutative, we have $n>1$.

    Let $\Theta$ be the set of all matrices $\left(
      \begin{array}{cc}
    A & B \\
    0&d\\
      \end{array}
    \right)\in M_n(D)$  where $A\in M_{n-1}(D)$, $B\in M_{(n-1)\times 1}(D)$ and 
    $d\in D$.
    Let $x=\left(
      \begin{array}{cc}
    A_1& B_1 \\
    0&d_1\\
      \end{array}
    \right),\,\,\,\,\, y=\left(
      \begin{array}{cc}
    A_2& B_2\\
    0&d_2\\
      \end{array}
    \right)\in \Theta$.
    Then $xy=\left(
      \begin{array}{cc}
    A_1A_2& B_3\\
    0&d_1d_2\\
      \end{array}
    \right)\in \Theta$, some $B_3 \in M_{(n-1)\times 1}(D)$.
    Hence $\Theta$ is a unitary subring of $M_n(D)$.
    Let $\beta$  be the set of all $\left(
      \begin{array}{cc}
    A & 0 \\
    0&1\\
      \end{array}
    \right)\in \Theta$. Then  $A\in M_{n-1}(D)^{*}$ if and only if 
    $\left(
      \begin{array}{cc}
    A & 0 \\
    0&1\\
      \end{array}
    \right)\in \Theta^*$.
    Thus $\Theta^*$ has a subgroup $\beta^*$ which is isomorphic to
    $M_{n-1}(D)^*$.
    
    Suppose $n\geq 4$ or suppose $n=3$ and $|D|>3$. Then $\Theta^*$ is not a solvable subgroup of $R^*=M_n(D)^*$, which is a contradiction, $R$ a minimal simple ring.
    
    Thus assume $n=3$ and $|D|\leq 3$. If $|D|=2$, then $R=M_3(GF(2))\in \Gamma$. If $|D|=3$, then $R=M_3(GF(3))\in \Gamma$.
    
    Suppose now $n=2$. Write $|D| = p^r$, $p$ a prime and $r$ a positive integer. If $r=1$, then clearly $p>3$, for otherwise $R^*$ would be solvable. Then $R\in \Gamma$. So suppose $r\neq 1$. Then $U=M_2(GF(p))$ is a proper unitary subring of $R=M_2(D)$.
    If $p>3$, then $U^*$ is not solvable, which is a contradiction. So $p\in \{2,3\}$. Suppose for a contradiction that $M_n(D)\not \in \Gamma$. Then $|D|\neq p^q$ for some prime $q$. Therefore $|D|=p^{rs}$ where $r,s\geq 2$. Then $D$ has a subfield $K$ of order $p^s$.
    So $M_2(K)$ is a proper unitary subring of $R$, which is a contradiction, as $M_2(K)^*$ is not a solvable group.
    
    {\bf Case 2.}
    Suppose now that $Jac(R)\neq 0$. 
    Let $\{M_1,...,M_k\}$ be the set of all maximal ideals of $R$. First suppose that $k>1$. Since $\frac{R}{Jac(R)}=\frac{R}{M_1\cap ...\cap M_k}\cong \frac{R}{M_1}\times ....\times \frac{R}{M_k}$, we have $(\frac{R}{Jac(R)})^*\cong (\frac{R}{M_1})^*\times ....\times (\frac{R}{M_k})^*$.
      If $(\frac{R}{M_i})^*$ is  a solvable group for all $i=1,2,...,k$, then $(\frac{R}{Jac(R)})^*$ is a solvable group.
      Since $1+Jac(R)$ is a $p$-group and $(\frac{R}{Jac(R)})^*\cong\frac{R^*+Jac(R)}{Jac(R)}$, we have $R^*$ is a solvable group, which is a contradiction. Hence, we may assume that $(\frac{R}{M_1})^*$ is not a solvable group.

       If $R/M_j\ncong GF(p)$ for some $1<j\leq k$, then $R/M_1\times (R/M_2)_0\times ...\times (R/M_k)_0$ is a proper subring of $\frac{R}{Jac(R)}$. Let $f:\frac{R}{Jac(R)}\cong \frac{R}{M_1}\times ....\times \frac{R}{M_k}$. Then $f^{-1}(R/M_1\times (R/M_2)_0\times ...\times (R/M_k)_0)$ is a proper unitary subring of $\frac{R}{Jac(R)}$ with non-solvable group of units (as $(\frac{R}{M_1})^*$ is not solvable by assumption), which is a contradiction. Hence $|R/M_j|=p$ for all $j=2,3,...,k$.
    Let $j>1$. Since $R/M_j$ is a simple $R$-module, there exists $e\in R$ such that $R/M_j=R/ann_R(e)$. Then  $M_j=ann_R(e)$ is two sided ideal. Clearly, we may assume that $e+M_j$ is the identity element of $R/M_j$.
    Since $ann_R(e)$ is a two side ideal, we have   $Re=\{0,e,2e,...,(p-1)e\}=eR$, and so 
     $Re$ is a two sided ideal. 
    Since $e\not\in Jac(R)$, there exists integer $n>1$ such that $e^n=e$.
    Let $ue\in ann_R(e)\cap Re$. Since $ann_R(e)\subseteq ann_R(e^{n-1})$, we have  
    $0=(ue)e^{n-1}=ue^n=ue$. So 
    $ann_R(e)\cap R_e=0$.
    Consequently, 
    $R= ann_R(e)\oplus Re$, which is a contradiction. So
    $k=1$, and so $R$ is a local ring.  
    Since $R$ is a finite ring there exists a positive integer $m$ such that 
    $Jac(R)^{m-1}\neq 0$ but $Jac(R)^m=0$.
    Let $I\subseteq Jac(R)^{m-1}$ be a minimal ideal of $R$.
    Clearly, $Jac(R)\subseteq ann_R(I)=\{x\in R: xa=0, \  \textit {for all }a \in I\}.$ Since $I^2=0$, $1+I$ is an elementary abelian $p$-group.
    We observe that $\frac{R}{I}$ does not have a proper subring $\frac{T}{I}$
    such that $(\frac{T}{I})^*$ is not a solvable group, for otherwise $R$ would have a proper unitary subring $T$ such that $T^*$ is not solvable, which is impossible.
    If $\frac{R}{I}=\frac{A}{I}\oplus \frac{B}{I}$, then $A\cup B\subseteq Jac(R)$, which is a contradiction, because $R$ is a local ring. 
    So $\frac{R}{I}\in \Gamma$, and then $I=Jac(R)$.
    But from   Lemma \ref{jj}, we have a contradiction.
    \endproof
    
    \begin{theorem}\label{2}
    
    Let $R$ be a finite ring of order $p^m$ where $p$ is a prime.
    
    (i) If $m\leq 5$ and $R^*$ is not a solvable group, then $R$ is a minimal simple ring.
    
    (ii) If $m=6$ and $R^*$ is not a solvable group, then $R$ is a minimal simple ring or $Jac(R)$ is the unique ideal of $R$ of size $p^2$ and $R$ has a subring $S\cong M_2(GF(p))$ such that $R=S\oplus_S Jac(R)$ as $S$-module.
    Also, $Jac(R)$ is a simple $S$-module.
    
    \end{theorem}
    
    \proof

    If $|R|=p^4$, then $R\cong M_2(GF(p))$ and thus $R$ is minimal simple (if $p>2$) or solvable (if $p=2$). So suppose that $p^5 \leq |R|\leq p^6$.
    If $Jac(R)=0$, then $R\cong M_n(GF(p))$ for some integer $n>1$.
    Since  $|R|>p^4$, we have $n>2$, and so $|R|>p^9$, which is a contradiction. So we must have $Jac(R)\neq 0$.
    Let $I$ be a minimal  ideal of $R$ such that $I\subseteq Jac(R)$.
    Since $1+Jac(R)$ is a $p$-group, $(\frac{R}{Jac(R)})^*$ is not a solvable group, so there exists 
    $x\in R^*$ such that $o(x)=p^2-1$ and $F:=\langle x+Jac(R)\rangle \cup \{Jac(R)\}$ is a field. We have $F^*=\langle x+Jac(R)\rangle$.
    Let $a\in I\setminus\{0\}$. Then $\{a,xa,x^2a,...,x^{p^2-1}a,0\}\subseteq I$.
    If there exist integers $1\leq i<j\leq p^2-1$ such that 
    $x^ia=x^ja$, then $(x^i-x^j)a=0$.  There exists 
    $x^m+Jac(R)\in F$ such that $x^i-x^j+Jac(R)=x^m+Jac(R)$.
    Since $Jac(R)a=0$, there exist $y,y'\in Jac(R)$, such that $0=(x^i-x^j)a=(x^i-x^j+y)a=(x^m+y')a$.
    Hence $x^ma=0$, and so $x^{-m}x^ma=a=0$, which is a contradiction. It follows that $|I|\geq p^2$.  
    
    Suppose $|R|=p^5$. Since $(\frac{R}{Jac(R)})^*$ is not a solvable group, we have $|\frac{R}{I}|\geq |\frac{R}{Jac(R)}|\geq p^4$, and so
    $|R|\geq p^4|I|\geq p^6$, which is a contradiction. This proves (i).
    
    Suppose that $|R|=p^6$. 
     If $Jac(R)\neq I$, then there exists $b\in Jac(R)\setminus I$.
     So $\{b+I,xb+I,...,x^{p^2-1}b+I,I\}\subseteq \frac{Jac(R)}{I}$. 
     Therefore   $|\frac{Jac(R)}{I}|\geq p^2$, and so $|R|\geq p^4|Jac(R)|=p^8$, which is a contradiction. Thus $Jac(R)=I$. 
    If $|I|>p^2$, then $(\frac{R}{I})^*$ is a solvable group, which is  a contradiction. So $|I|=p^2$.  Hence $I=\{a,xa,x^2a,...,x^{p^2-1}a,0\}=\{a,ax,...,ax^{p^2-1},0\}$.
    Let $S$ be a unitary maximal subring of $R$ such that $S^*$ is not a solvable group. Since  $S$ is a minimal simple ring, then $Jac(S)=0$, and so
    $I\cap S=0$. Since $S$ is a maximal unitary subring, $S+I=R$. 
    Clearly $I$ is a $S$-module, so $R=S\oplus_S I$.
    If $I$ has proper $S$-submodule $D\neq 0$, then $S+D$ is a proper unitary subring of $R$ such that $(S+D)^*$ is not a solvable group.
    Clearly, $|D| \leq |Jac(S+D)|$. But by (i) $Jac(S+D)=0$, which is a contradiction. So $I$ does not have proper $S$-submodule, and so $I$ is a simple $S$-module. This proves (ii).
    \endproof
    
    \addvspace{\bigskipamount}

    \end{document}